\documentclass{article}

\usepackage[english]{babel}
\usepackage{amsthm}
\usepackage{amsmath}
\usepackage{amssymb}

\usepackage{graphicx}
\usepackage{enumitem}
\usepackage[colorlinks]{hyperref}
\usepackage{caption}
\usepackage{subcaption}
\usepackage{pict2e,color,colordvi,amscd}
\usepackage{indentfirst}
\usepackage{tikz}
\usepackage{float}

\newtheorem{definition}{Definition}[section]
\newtheorem{theorem}[definition]{Theorem}
\newtheorem{lemma}[definition]{Lemma}
\newtheorem{claim}{Claim}[section]
\newtheorem{corollary}[definition]{Corollary}

\newtheorem{conjecture}{Conjecture}

\newtheorem{observation}{Observation}

\title{Average degrees of edge-$\Delta$-critical multigraphs}
\date{\relax}
\author{
Guantao Chen\thanks{
Department of Mathematics and Statistics, Georgia State University, Atlanta, GA 30303.
Email addresses: \texttt{\{gchen, yma29, ysu12, swang85\}@gsu.edu}.}
\and
Yuying Ma\footnotemark[1]
\and
Yimo Su\footnotemark[1]
\and
Shengze Wang\footnotemark[1]
}

\begin{document}
\maketitle
\begin{abstract}
Let $G$ be a loopless multigraph with maximum degree $\Delta(G)$, average degree $\overline{d}(G)$, density $\Gamma(G)$, and chromatic index $\chi'(G)$. 
A multigraph $G$ is called edge-$\Delta$-critical if $\Delta(G)=\Delta$, $\chi'(G)=\Delta(G)+1$ and $\chi'(H) \le \Delta(G)$ for every proper subgraph $H\subset G$. Vizing conjectured that if $G$ is an edge-$\Delta$-critical simple graph on $n$ vertices, then $\overline{d}(G) \ge \Delta-1+\tfrac{3}{n}$.
Motivated by this, we conjecture that every edge-$\Delta$-critical multigraph $G$ satisfies $\overline{d}(G) \ge \tfrac{2\Delta+2}{3}$, which is best possible. 
We first give a general lower bound in this direction. For any such graph $G$, 
\[
\overline{d}(G) \ge 
\begin{cases}
\frac{\sqrt{17}-3}{2}(\Delta+1) & \text{if } \Delta \le 112;\\[4pt]
\frac{\Delta+\sqrt{2\Delta-1}}{2} & \text{if } \Delta \ge 113.
\end{cases}
\]
This bound can be further improved under an additional condition on the multiplicity $\mu$. In this case,
\[
\overline{d}(G)\ge 
\min\left\{
\frac{2\mu\Delta+2\mu(2\mu-1)}{4\mu-1},\;
\frac{\sqrt{17}-3}{2}(\Delta+1)
\right\}.
\]
We also confirm the conjecture for $\Delta \in \{2,3,4,5,6,7,8\}$. As a consequence, Goldberg's conjecture~\cite{Goldberg1984} holds for $\Delta(G)\in\{2,3,4,5\}$, that is, every multigraph $G$ with $\chi'(G)\ge \Delta(G)+1$ satisfies $\Gamma(G)\ge \Delta(G)$.

\vspace{0.3cm}
\end{abstract}

\section{Introduction}\label{sec:intro}

All graphs considered in this paper are finite, undirected, and loopless, and may contain parallel edges. We will state explicitly when a graph is simple. For $k \in \mathbb{N}$, we define $[k]:= \{1,2,\dots, k\}$. Let $G$ be a graph with vertex set $V(G)$ and edge set $E(G)$. We denote its multiplicity, minimum degree, average degree, and maximum degree by $\mu(G)$, $\delta(G)$, $\overline{d}(G)$, and $\Delta(G)$, respectively.

A (proper) $k$-edge-coloring of a graph $G$ is a mapping $\varphi: E(G) \to [k]$ such that the preimage $\varphi^{-1}(i)$ is a matching (that is, a set of edges no two of which share a common vertex) for each $i \in [k]$. Each set $\varphi^{-1}(i)$ is called a \emph{color class} of $\varphi$. Thus a $k$-edge-coloring of $G$ partitions $E(G)$ into $k$ matchings. Let $\mathcal{C}^k(G)$ denote the set of all $k$-edge-colorings of $G$. The \emph{chromatic index} $\chi'(G)$ is the minimum integer $k$ such that $\mathcal{C}^k(G) \ne \emptyset$. 
A graph $G$ is called \emph{critical} if $\chi'(H)<\chi'(G)$ for every proper subgraph $H \subset G$, and \emph{(edge-)$k$-critical} if it is critical and $\chi'(G) = k+1$. When $k=\Delta(G)=\Delta$, $G$ is called \emph{(edge-)$\Delta$-critical}. By definition, $\chi'(G)\ge \Delta(G)$. Shannon~\cite{Shannon1949} proved that $\chi'(G) \le \lfloor \tfrac 32 \Delta(G) \rfloor$, and Vizing~\cite{vizing1964} showed that $\chi'(G) \le \Delta(G)+\mu(G)$. In particular, for simple graphs, $\chi'(G)\in \{\Delta(G),\,\Delta(G)+1\}$. 

It is well known that every simple graph with $\chi'(G)=\Delta(G)+1$ contains a $\Delta$-critical subgraph with the same maximum degree. The structure of $\Delta$-critical simple graphs has drawn much attention. Vizing~\cite{vizing1968} conjectured the following for such graphs, known as Vizing's Average Degree Conjecture.
\begin{conjecture}[Vizing~\cite{vizing1968}]
    Every $\Delta$-critical simple graph $G$ satisfies $\overline{d}(G) \ge \Delta-1+\frac{3}{|G|}$.
\end{conjecture} 
A number of results improve lower bounds, but the conjectured bound remains open (see~\cite{CC2021299, CCJLL20191613, woodall2007average}).

A natural question is whether a similar result holds for $\Delta$-critical graphs. Motivated by this, we propose the following conjecture.

\begin{conjecture}\label{conj:avg-degree}
    Every $\Delta$-critical graph $G$ with $\Delta \ge 2$ satisfies $\overline{d}(G) \ge \frac{2\Delta+2}{3}$.
\end{conjecture}
This bound is best possible.
Indeed, consider the graph obtained from $K_3$ by replacing two edges with $k$ parallel edges each, where $k \ge 1$. In this graph, we have $\Delta (G) = 2k$ and $\chi'(G) = 2k + 1$. Also, removing any edge will decrease the chromatic index. Hence $G$ is $\Delta$-critical, and its average degree is $\overline{d}(G) = \frac{2(2k+1)}{3} = \frac{2\Delta(G)+2}{3}$, which exactly matches the bound.

In this paper, we obtain two main results. The first gives a general lower bound on the average degree of $\Delta$-critical graphs, and a stronger bound under an additional condition on the multiplicity. 
For the latter, we introduce and apply a generalized version of Woodall's adjacency lemma~\cite{woodall2007average}.
\begin{theorem}\label{THM-General}
    If $G$ is a $\Delta$-critical graph, then
    $\overline{d}(G)\geq \min\left\{\frac{\Delta+\sqrt{2\Delta-1}}{2}, \frac{\sqrt{17}-3}{2}(\Delta+1)\right\}$, that is,
    \[
    \overline{d}(G) \ge 
    \begin{cases}
    \frac{\sqrt{17}-3}{2}(\Delta+1) & \text{if } \Delta \le 112;\\[4pt]
    \frac{\Delta+\sqrt{2\Delta-1}}{2} & \text{if } \Delta \ge 113.
    \end{cases}
    \]
    Moreover, if its multiplicity satisfies
    $$
    \mu(G)\le
    \max\left\{
    \frac{1+\sqrt{2\Delta-1}}4,\,
    \frac{1}{2}+\frac{\sqrt{2}}{4}\sqrt{(5-\sqrt{17})\Delta+3-\sqrt{17}}\right\},
    $$
    then
    $$
    \overline{d}(G)\ge \min\left\{\frac{2\mu\Delta+2\mu(2\mu-1)}{4\mu-1}, \frac{\sqrt{17}-3}{2}(\Delta+1)\right\}.
    $$
\end{theorem}
The first bound holds without any restriction on the multiplicity and yields $\overline{d}(G) \ge \tfrac{1}{2}\Delta$, whereas the second requires a condition on $\mu$ but ensures $\overline{d}(G) > \tfrac{1}{2}\Delta$.

Second, we verify Conjecture~\ref{conj:avg-degree} for small values of $\Delta(G)$.
\begin{theorem}\label{thm:avg-degree-2-8}
Let $G$ be a $\Delta$-critical graph. If $\Delta \in \{2,3,4,5,6,7,8\}$, then $\overline{d}(G)\geq \frac{2\Delta+2}{3}$.
\end{theorem}

Define the density of a graph $G$ by
\[
\Gamma(G):=\max_{H\subseteq G,\ |H|\ge2}
\left\lceil\frac{|E(H)|}{\lfloor |H|/2\rfloor}\right\rceil.
\]
It is well known that $\chi'(G)\ge \max\{\Delta(G),\Gamma(G)\}$. The Goldberg--Seymour Conjecture~\cite{Goldberg1973,Seymour1979}, recently confirmed~\cite{CHYZ2024,CJZ2025,J2026}, yields $\chi'(G)\in \{\max\{\Delta(G),\Gamma(G)\},\ \max\{\Delta(G)+1,\Gamma(G)\}\}$.

In 1984, Goldberg~\cite{Goldberg1984} further proposed the following conjecture on the density of graphs, which we refer to as Goldberg's Density Conjecture.
\begin{conjecture}[Goldberg~\cite{Goldberg1984}]\label{conj:density}
    For any graph $G$, if $\Gamma(G) \le \Delta(G)-1$, then $\chi'(G) = \Delta(G)$.
\end{conjecture}
Equivalently, if $\chi'(G) \ge \Delta(G)+1$, then $\Gamma(G) \ge \Delta(G)$. By the confirmed Goldberg--Seymour Conjecture, if $\chi'(G)\ge \Delta(G)+2$, then $\Gamma(G)= \chi'(G)\ge \Delta(G)+2 > \Delta(G)$. 
The nontrivial case is $\chi'(G) = \max\{\Delta(G)+1, \Gamma(G)\} = \Delta(G)+1$. For $\Delta$-critical graphs with $\Delta \ge 3$, the density exceeds the average degree (see Corollary~\ref{cor: density>degree}). 
Therefore, Theorem~\ref{thm:avg-degree-2-8} verifies Conjecture~\ref{conj:density} for $\Delta(G) \in \{2,3,4,5\}$.

The rest of the paper is organized as follows. In the next section, we introduce notation and preliminary results on coloring and $\Delta$-critical graphs. Section~\ref{sec:general} presents the proof of Theorem~\ref{THM-General}. Section~\ref{sec:2-8} proves Theorem~\ref{thm:avg-degree-2-8} for small values of $\Delta$. A generalized version of Woodall's adjacency lemma is stated in Section~\ref{sec:pre} and proved in Section~\ref{sec:prooflemma}. We conclude this paper with some remarks.

We use the following notation throughout the paper. Let $G$ be a graph. For $X, Y\subseteq V(G)$, let $E_G(X, Y)$ denote the set of edges with one end in $X$ and the other in $Y$. We write $E_G(x):=E_G(\{x\},V(G))$ and $E_G(x,y):=E_G(\{x\},\{y\})$. The multiplicity of $\{x,y\}$ is $\mu_G(xy)=|E_G(x,y)|$.
Let $N_G(x):=\{y\in V(G): E_G(x,y)\ne\emptyset\}$ be the neighborhood of $x$, and let $d_G(x):=|E_G(x)|$ be the degree of $x$. A vertex of degree $d$ is called a \emph{$d$-vertex}, and a \emph{$d$-neighbor} of $x$ is a neighbor of $x$ that is a $d$-vertex. For $F \subseteq E(G)$, let $V_G(F)$ denote the set of endvertices of edges in $F$, and write $V_G(e)$ when $F=\{e\}$.
When there is no ambiguity, we omit the subscript $G$. For $X\subseteq V(G)$ and $F\subseteq E(G)$, let $G-X$ denote the subgraph with $V(G-X)=V(G)\setminus X$ and $E(G-X)=E(G)\setminus E_G(X,V(G))$, and let $G-F$ denote the subgraph with $V(G-F)=V(G)$ and $E(G-F)=E(G)\setminus F$. When $X=\{x\}$ and $F=\{e\}$, we write $G-\{x\}$ simply as $G-x$, and $G-\{e\}$ as $G-e$.

\section{Preliminaries}\label{sec:pre}
In this section, we introduce some notation and collect some definitions and results on edge-coloring and $\Delta$-critical graphs.
We will generally follow~\cite{SSTF2012GraphEC} for notation and definitions. 

Let $G$ be a $\Delta$-critical graph, $e_1 \in E_G(y_0, y_1)$, and let $\varphi \in \mathcal{C}^{\Delta}(G-e_1)$. For each vertex $v \in V(G)$, define 
$$
\varphi(v) = \{\varphi(e): e\in E_G(v)\setminus\{e_1\}\} \text{ and } \overline{\varphi}(v) = [\Delta]\setminus\varphi(v).
$$
We refer to $\varphi(v)$ as the set of colors \emph{present} at $v$, and to $\overline{\varphi}(v)$ as the set of colors \emph{missing} at $v$. We say a set $X \subseteq V(G)$ is \emph{$\varphi$-elementary} if the sets $\overline{\varphi}(v)$ with $v \in X$ are pairwise disjoint. For distinct colors $\alpha, \beta \in [\Delta]$, let $H$ be the subgraph of $G$ with $V(H) = V(G)$ and $E(H) = \varphi^{-1}(\alpha) \cup \varphi^{-1}(\beta)$. Each component of $H$ is either a path or an even cycle (possibly a $2$-cycle), whose edges are colored alternately with $\alpha$ and $\beta$. Such a component is called an \emph{$(\alpha,\beta)$-chain} of $G$ with respect to $\varphi$. Any two $(\alpha,\beta)$-chains are either identical or disjoint. 

For a multigraph, the endvertices $x,y$ may not determine the edges, but if the coloring of the edge is clear or not important, we still write $e=xy$ for brevity.

A \emph{multi-fan} introduced by Stiebitz et al.~\cite{SSTF2012GraphEC} at $y_0$ with respect to $e_1\in E_G(y_0, y_1)$ and $\varphi \in \mathcal{C}^{\Delta}(G-e_1)$ is a sequence 
$$
F=(e_1, y_1, e_2, y_2,\dots, e_p, y_p),
$$ 
where $p \ge 1$, consisting of edges $e_1, e_2, \dots, e_p$ and vertices $y_1, y_2, \dots, y_p$, such that:
\begin{itemize}
    \item[(F1)] The edges $e_1, e_2, \dots, e_p$ are distinct, and $e_i \in E_G(y_0, y_i)$ for each $i \in [p]$.
    \item[(F2)] For each $2 \le i \le p$, there exists $1 \le j < i$ such that $\varphi(e_i) \in \overline{\varphi}(y_j)$.
\end{itemize}
Note that $y_0$ is not in $V(F)$. Stiebitz et al. studied multi-fans and obtained the following results. 

\begin{lemma}[Stiebitz, Scheide, Toft and Favrholdt~\cite{SSTF2012GraphEC}]\label{lem:multi_fan} 
    Let $F = (e_1, y_1, \dots, e_p, y_p)$ be a multi-fan at $y_0$ with respect to $e_1$ and $\varphi$. The following statements hold:
    \begin{enumerate}[label=(\alph*)]
        \item $\{y_0, y_1, \dots, y_p\}$ is $\varphi$-elementary;
        \item If $\alpha \in \overline{\varphi}(y_0)$ and $\beta \in \overline{\varphi}(y_i)$ for $i \in [p]$, then there is an $(\alpha, \beta)$-chain with respect to $\varphi$ having endvertices $y_0$ and $y_i$;
        \item If $F$ is a maximal multi-fan at $y_0$ with respect to $e_1$ and $\varphi$, then $|V(F)| \ge 2$ and 
        \[
        \sum_{z \in V(F)}(d_G(z)+\mu_F(y_0z)-\Delta) = 2.
        \]
    \end{enumerate}
\end{lemma}

A \emph{Kierstead path}, first introduced by Kierstead~\cite{K1984}, with respect to $e_1 \in E_G(y_0, y_1)$ and $\varphi \in \mathcal{C}^\Delta(G-e_1)$ is a sequence 
$$K = (y_0, e_1, y_1, \dots, e_p, y_p),$$ where $p \ge 1$, consisting of edges $e_1, e_2, \dots, e_p$ and vertices $y_0, y_1, \dots, y_p$, such that:
\begin{itemize}
    \item[(K1)] The vertices $y_0, y_1, \dots, y_p$ are distinct, and $e_i \in E_G(y_{i-1}, y_i)$ for each $i \in [p]$.
    \item[(K2)] For each $2 \le i \le p$, there exists $0 \le j < i$ such that $\varphi(e_i) \in \overline{\varphi}(y_j)$.
\end{itemize}
Clearly, a Kierstead path on three vertices is a multi-fan at $y_1$ and thus $\varphi$-elementary. Kostochka and Stiebitz proved the following results concerning Kierstead paths on four vertices.

\begin{lemma}[Kostochka and Stiebitz~\cite{SSTF2012GraphEC}]\label{lem:ShortK_Path}
    Let $K=(y_0, e_1, y_1, e_2, y_2, e_3, y_3)$ be a Kierstead path on four vertices with respect to $e_1$ and $\varphi$. Then, the following hold:
    \begin{enumerate}[label=(\alph*)]
        \item If $\min\{d_G(y_1),\  d_G(y_2)\}<\Delta$, then $V(K)$ is $\varphi$-elementary;
        \item $|\overline{\varphi}(y_3)\cap (\overline{\varphi}(y_0)\cup \overline{\varphi}(y_1))|\le 1$;
        \item $d_G(y_1)+d_G(y_2)+d_G(y_3) \ge 2\Delta+2$, and moreover $d_G(y_0)+d_G(y_1)+d_G(y_2)+d_G(y_3) \ge 3\Delta+1$ with equality only if $\min\{d_G(y_1),\  d_G(y_2)\}=\Delta$.
    \end{enumerate}
\end{lemma}

We now derive a useful consequence of Lemmas~\ref{lem:multi_fan}
and~\ref{lem:ShortK_Path}.

Let $G$ be a $\Delta$-critical graph, let $e \in E_G(x,y)$, and let $\varphi \in \mathcal{C}^{\Delta}(G-e)$.
Since $\overline{\varphi}(x)\cap \overline{\varphi}(y)=\emptyset$, the vertex $y$
is incident with exactly $\Delta+1-d_G(x)$ edges whose colors lie in
$\overline{\varphi}(x)$. Let
$$
F_y(\varphi)=\{f\in E_G(y)\setminus E_G(x,y): \varphi(f)\in \overline{\varphi}(x)\},
$$
and let
$$
Z_y(\varphi)=V_G(F_y(\varphi))\setminus\{y\}.
$$
Clearly, $Z_y(\varphi)\subseteq N_G(y)\setminus\{x\}$ and $|E_G(y, Z_y(\varphi))| \ge |F_y(\varphi)| = |\overline{\varphi}(x)|= \Delta+1-d_G(x)$.

\begin{lemma}\label{lem:degree}
Let $F_y(\varphi)$ and $Z_y(\varphi)$ be defined as above. Then, the following hold:
\begin{enumerate}[label=(\alph*)]
 \item For each $z\in Z_y(\varphi)$, $z$ is incident
    with at least $2\Delta+2-d_G(x)-d_G(y)$ edges $zu$. Moreover, for any such edge
    $zu$ with $u\notin\{x,y\}$,
    \[
    d_G(u)\ge
    \begin{cases}
    2\Delta+2-d_G(x)-d_G(y) & \text{if } \min\{d_G(y),d_G(z)\}<\Delta,\\[4pt]
    2\Delta+1-d_G(x)-d_G(y) & \text{if } \min\{d_G(y),d_G(z)\}=\Delta;
    \end{cases}
    \]
    \item
    \[
    \sum_{z\in Z_y(\varphi)} d_G(z)\ge
    \Delta |Z_y(\varphi)|+\Delta+2-d_G(x)-d_G(y).
    \]
\end{enumerate}
Similarly, by interchanging the roles of $x$ and $y$, the same conclusions hold.
\end{lemma}

\begin{proof}
Let $\overline{\varphi}(x)=\{\alpha_1,\dots,\alpha_p\}$, where $p = \Delta+1-d_G(x)$. By the definition of $F_y(\varphi)$, we have $\varphi(F_y(\varphi)) =\overline{\varphi}(x)$. Hence, for each $i\in [p]$, let $e_i=yy_i$ be the unique edge in $F_y(\varphi)$ with $\varphi(e_i)=\alpha_i$. We note that the vertices $y_i$ are not necessarily distinct. Then
\[
F=(e,x,e_1,y_1,\dots,e_p,y_p)
\]
forms a multi-fan at $y$ with respect to $e$ and $\varphi$. Clearly, $Z_y(\varphi) \subseteq V(F)$.
By Lemma~\ref{lem:multi_fan}$(a)$, $Z_y(\varphi) \cup \{x, y\}$ is $\varphi$-elementary, and thus, for all $i \in [p]$,
    \begin{equation}\label{eq:disjoint}
         \overline{\varphi}(y_i) \cap (\overline{\varphi}(x) \cup \overline{\varphi}(y)) = \emptyset,
    \end{equation}
    and
    \begin{equation}\label{eq:sum}
        \sum_{z\in Z_y(\varphi)}|\overline{\varphi}(z)|+|\overline{\varphi}(x)|+|\overline{\varphi}(y)|\le \Delta.
    \end{equation}

   We first show $(a)$ using~\eqref{eq:disjoint}.
   It follows that, for all $i\in [p]$, $y_i$ is incident with $|\overline{\varphi}(x)|+|\overline{\varphi}(y)| = 2\Delta+2-d_G(x)-d_G(y)$ edges colored with colors missing at $x$ or $y$. For any such edge $f = y_iu$ with $u \notin \{x,y\}$, $K = (x, e, y, e_i, y_i, f, u)$ forms a Kierstead path on four vertices with respect to $\varphi$ and $e$. Applying Lemma~\ref{lem:ShortK_Path} to $K$, we have $$
    |\overline{\varphi}(u) \cap (\overline{\varphi}(x) \cup \overline{\varphi}(y))| \le 1,
    $$ 
    that is, $d_G(u) \ge 2\Delta + 1 - d_G(x)-d_G(y)$. 
    Moreover, if $\min\{d_G(y), d_G(y_i)\}< \Delta$, then $V(K)$ is $\varphi$-elementary, which implies $d_G(u) \ge 2\Delta + 2 - d_G(x) - d_G(y)$. So, $(a)$ holds. 
    
    Next we show $(b)$ using~\eqref{eq:sum}. Substituting $|\overline{\varphi}(z)| = \Delta - d_G(z)$ for $z \in Z_y(\varphi)$, $|\overline{\varphi}(x)|= \Delta - d_G(x)+1$, and $|\overline{\varphi}(y)| = \Delta-d_G(y)+1$ into~\eqref{eq:sum}, we obtain
   $$
   \Delta|Z_y(\varphi)|-\sum_{z\in Z_y(\varphi)}d_G(z)+2\Delta+2-d_G(x)-d_G(y) \le \Delta,
   $$
   which implies
   $$
   \sum_{z\in Z_y(\varphi)}d_G(z) \ge \Delta|Z_y(\varphi)|+\Delta+2-d_G(x)-d_G(y),
   $$ as desired.
\end{proof}

In the proof of Theorem~\ref{thm:avg-degree-2-8}, we will not use Lemma~\ref{lem:degree}, but only the following corollary of it.

\begin{corollary}\label{cor:degree}
Let $G$ be a $\Delta$-critical graph and $x, y \in V(G)$ with $E_G(x, y) \ne \emptyset$. Then, there exists a nonempty set $Z_y\subseteq N_G(y)\setminus\{x\}$ with $|E_G(y, Z_y)|\ge \Delta+1-d_G(x)$ such that:
\begin{enumerate}[label=(\alph*)]
    \item For each $z\in Z_y$, $z$ is incident with at least $2\Delta+2-d_G(x)-d_G(y)$ edges $zu$, and for any such edge $zu$ with 
    $u\notin\{x,y\}$, we have
    \[
    d_G(u) \ge 
    \begin{cases}
    2\Delta+2-d_G(x)-d_G(y) & \text{if } \min\{d_G(y), d_G(z)\}<\Delta,\\[4pt]
    2\Delta+1-d_G(x)-d_G(y) & \text{if } \min\{d_G(y), d_G(z)\}=\Delta;
    \end{cases}
    \]
    \item 
    \[
    \sum_{z\in Z_y} d_G(z) \ge \Delta|Z_y|+\Delta+2-d_G(x)-d_G(y).
    \]
\end{enumerate}
Similarly, there exists a nonempty set 
$Z_x \subseteq N_G(x)\setminus\{y\}$ with $|E_G(x, Z_x)| \ge \Delta+1-d_G(y)$ satisfying the same conclusions with the roles of $x$ and $y$ interchanged.
\end{corollary}

For an ordered pair $(x,y)$ of distinct vertices in a graph $G$, the \emph{Kierstead set} of $(x,y)$,
introduced by Stiebitz et al.~\cite{SSTF2012GraphEC}, is
$$
K_G(x,y)=\{z\in N_G(y)\setminus\{x\} :  d_G(x)+d_G(y)+d_G(z)\ge2\Delta(G)+2\},
$$ 
and the corresponding \emph{Kierstead number} is $\sigma_G(x,y)=|K_G(x,y)|$. We note that if $G$ is $\Delta$-critical and $E_G(x, y) \ne \emptyset$, then for any $e \in E_G(x,y)$ and $\varphi \in \mathcal{C}^{\Delta}(G-e)$, $Z_y(\varphi)$ and $Z_x(\varphi)$ are contained in $K_G(x,y)$ and $K_G(y,x)$, respectively.
The Kierstead number for simple graphs was first defined and studied by Woodall~\cite{woodall2007average}.  
We recall the following result. 
\begin{lemma}[Woodall~\cite{woodall2007average}]
    Let $G$ be a $\Delta$-critical simple graph and $xy \in E(G)$. Then there are at least $\Delta-\sigma(x,y) \ge \Delta-d_G(y)+1$ vertices $z\in K_G(y,x)$ such that $\sigma(x,y)+\sigma(x,z) \ge 2\Delta-d_G(x)$.
\end{lemma}

To deal with multigraphs, we introduce a modified version of the Kierstead number by counting edges instead of vertices, defined as follows.
$$
\sigma'_G(x, y) = |E_G(y,K_G(x,y))|.
$$
For simple graphs, $\sigma'_G(x,y) = \sigma_G(x,y)$. 
The following lemma generalizes Woodall’s result and is key to the proof of the second part of Theorem~\ref{THM-General}. Due to its length, the proof is deferred to Section~\ref{sec:prooflemma}.

\begin{lemma}\label{lem: sigma}
Let $G$ be a $\Delta$-critical graph with multiplicity $\mu$ and $x, y\in V(G)$ with $E_G(x,y)\neq\emptyset$. Then, there exists a set $Z \subseteq K_G(y,x)$ with
$|E_G(x, Z)| \ge \Delta-\sigma'(x,y)-\mu_G(xy)+1 \ge \Delta-d_G(y)+1$ such that, for every $z\in Z$, $\sigma'(x,y)+\sigma'(x,z) \ge 2\Delta-d_G(x)+2-\mu_G(xy)-\mu_G(xz) \ge 2\Delta-d_G(x)-2\mu+2$.
\end{lemma}

We now focus on some basic structural properties of $\Delta$-critical graphs.
\begin{observation}\label{obs:critical}
    Every $\Delta$-critical graph $G$ is $2$-connected. In particular, its underlying simple graph is bridgeless and has minimum degree at least $2$.
\end{observation}
\begin{proof}
    Clearly, $\Delta \ge 2$ and $|G|\ge 3$. 
    Suppose that $G$ has a cut vertex $v$, and that $G-v$ has components $G_1,\dots,G_{\ell}$, where $\ell \ge 2$. By the definition of $\Delta$-criticality, each subgraph $G[V(G_i)\cup\{v\}]$ is $\Delta$-edge-colorable. For each $i\in[\ell]$, fix a $\Delta$-edge-coloring $\varphi_i$ of $G[V(G_i)\cup\{v\}]$. Since $\sum_{i=1}^{\ell} |\varphi_i(v)| = d_G(v) \le \Delta$, by permuting colors if necessary, we may assume that the sets $\varphi_1(v), \dots, \varphi_{\ell}(v)$ are pairwise disjoint. This yields a $\Delta$-edge-coloring of $G$, a contradiction.
\end{proof}

Beineke and Fiorini~\cite{Beineke1976} proved that a $\Delta$-critical simple graph is not regular for $\Delta \ge 3$. The condition $\Delta \ge 3$ is necessary since every odd cycle is $2$-regular and $2$-critical. However, the situation for multigraphs is more subtle, since the chromatic index of a multigraph is not restricted to the values $\Delta$ and $\Delta + 1$. In fact, an $r$-critical multigraph with $r > \Delta$ may be regular. For example, let $G$ be the multigraph obtained from $K_3$ by replacing each edge with $k$ parallel edges, where $k \ge 2$. It is easy to see that $G$ is $(3k-1)$-critical and $(2k)$-regular.

Nevertheless, the same conclusion holds for $\Delta$-critical graphs.

\begin{lemma}\label{lem:no-regular}
    Every $\Delta$-critical graph $G$ with $\Delta \ge 3$ is not regular.
\end{lemma}
\begin{proof}
Suppose for a contradiction that $G$ is regular. Let $e \in E_G(x,y)$ for some $x, y \in V(G)$, and let $\varphi \in \mathcal{C}^{\Delta}(G-e)$. Then, $d_{G-e}(u)=\Delta-1$ for $u\in\{x,y\}$ and $d_{G-e}(u)=\Delta$ otherwise. Hence $|\overline{\varphi}(u)|=1$ for $u\in\{x,y\}$ and $0$ otherwise. Since $\overline{\varphi}(x)\cap \overline{\varphi}(y)=\emptyset$, write $\overline{\varphi}(x)=\{\alpha\}$ and $\overline{\varphi}(y)=\{\beta\}$ with $\alpha\ne\beta$. The color classes $\varphi^{-1}(\alpha)$ and $\varphi^{-1}(\beta)$ are near-perfect matchings of $G-e$, missing $x$ and $y$, respectively. Hence $|G|$ is odd. Since $\Delta \ge 3$, choose $\gamma \in [\Delta]\setminus\{\alpha, \beta\}$. Then, $\gamma$ is present at every vertex of $G-e$, so $\varphi^{-1}(\gamma)$ would be a perfect matching of the odd-order graph $G$, a contradiction.
\end{proof}

We conclude this section with some results relating density and average degree.

\begin{lemma}\label{lem:density}
    If $G$ is a graph with $\Gamma(G) = \overline{d}(G)$, then $|G|$ is even and $G$ is regular.
\end{lemma}
\begin{proof}
Suppose first that $|G|$ is odd. Then
\[
\Gamma(G) \ge \frac{|E(G)|}{\lfloor|G|/2\rfloor} = \frac{2|E(G)|}{|G|-1} > \frac{2|E(G)|}{|G|} = \overline{d}(G),
\]
which is a contradiction. Thus $|G|$ is even.

Next we show that $G$ is regular. Let $v$ be a vertex such that $d_G(v) = \delta(G)$ and set $H := G-v$. Clearly, $\Gamma(G) \ge \Gamma(H)$. Since $|H|$ is odd, we have
$\Gamma(H) \ge \frac{2|E(H)|}{|H|-1} = \frac{2(|E(G)|-\delta(G))}{|G|-2}$. We claim that 
$$
\frac{2(|E(G)|-\delta(G))}{|G|-2} \ge \frac{2|E(G)|}{|G|}=\overline{d}(G).
$$
Indeed, otherwise multiplying both sides by $|G|(|G|-2)$ yields $2|E(G)| < |G|\, \delta(G)$, that is, $\delta(G) > \frac{2|E(G)|}{|G|} = \overline{d}(G)$, a contradiction. Thus $\Gamma(G) \ge \Gamma(H) \ge \overline{d}(G)$. Since by hypothesis $\Gamma(G) = \overline{d}(G)$, equality holds throughout, and hence $\delta(G) = \overline{d}(G)$. Therefore $G$ is regular.
\end{proof}

The following is deduced from Lemmas~\ref{lem:no-regular} and~\ref{lem:density}.
\begin{corollary}\label{cor: density>degree}
    Every $\Delta$-critical graph $G$ with $\Delta \ge 3$ satisfies $\Gamma(G) > \overline{d}(G)$.
\end{corollary}

\section{Proof of Theorem~\ref{THM-General}}\label{sec:general}
We begin with the proof of the general bound, and then consider the case when the multiplicity is relatively small.

\noindent \textbf{Proof of Theorem~\ref{THM-General}.}
   The proof uses a discharging method based on the vertex degrees of $G$. Let $q=\min\Big\{\tfrac{\Delta+\sqrt{2\Delta-1}}{2}, \tfrac{\sqrt{17}-3}{2}(\Delta+1)\Big\}$. Assign to each vertex $x$ of $G$ an initial charge $M(x)=d_G(x)$. For each vertex $y \in V(G)$, let $E_y$ be the set of edges $e \in E_G(x,y)$ with $M(x)< q$, and let $m_y :=|E_y|$. We redistribute the charge according to the following rule:
    \begin{itemize}
        \item \textbf{Rule.} Each vertex $y \in V(G)$ with $M(y) > q$ and $m_y \ge 1$ sends $\mu_G(xy)\frac{M(y)-q}{m_y}$ to each neighbor $x$ with $M(x)<q$. 
    \end{itemize}
    Denote by $M'(x)$ the resulting charge on each vertex $x$. For any vertex $y$ with $M(y) \ge q$, we have $M'(y) \ge q$; moreover, $M'(y) = M(y)-m_y \cdot \tfrac{M(y)-q}{m_y} = q$ if $m_y \ne 0$. It suffices to consider vertices of degree less than $q$ and show that, after the discharging process, each such vertex has charge at least $q$. Let $x\in V(G)$ such that $d_G(x) <q$.

\noindent \textbf{Case 1.} $d_G(x)\leq \Delta-q+1$.
Let $y\in N_G(x)$. Then $d_G(y) \ge q+1$. By Corollary~\ref{cor:degree}, $y$ is incident with at least $\Delta+1-d_G(x)$ edges whose other endvertex has degree at least $2\Delta+2-d_G(x)-d_G(y)\ge q+1$. Thus $m_y \leq d_G(y)-(\Delta +1 -d_G(x))=d_G(y)+d_G(x)-\Delta -1$. 
Each $y$ sends to $x$ $$\mu_G(xy)\cdot \frac{M(y)-q}{m_y}\geq \mu_G(xy)\cdot \frac{d_G(y)-q}{d_G(y)+d_G(x)-\Delta -1}\geq \mu_G(xy)\cdot \frac{\Delta-q}{d_G(x)-1}.$$ The last inequality follows from the fact that the function $\frac{d_G(y)-q}{d_G(y)+d_G(x)-\Delta-1}$ is non-increasing in $d_G(y)$, together with $d_G(y)\leq \Delta$. Therefore,
\begin{align*}
   M'(x)&\geq M(x)+ d_G(x)\cdot\frac{\Delta-q}{d_G(x)-1}\\
   &= \Delta-q+1+(d_G(x)-1)+\frac{\Delta-q}{d_G(x)-1}\\
   & \geq \Delta-q+1+2\sqrt{\Delta-q}\\
   & \geq q,
\end{align*}
where the last inequality follows from the assumption $q\le \frac{\Delta+\sqrt{2\Delta-1}}{2}$.

\noindent \textbf{Case 2.} $\Delta-q+1 < d_G(x)< q$. Choose $y\in N_G(x)$ with $d_G(y)$ minimum. We consider two subcases.

\noindent \textbf{Subcase 2.1.} $d_G(y)\leq q$.
By Corollary~\ref{cor:degree}, $x$ is incident with at least $\Delta-d_G(y)+1\geq \Delta -q+1$ edges whose other endvertex has degree at least $2\Delta+2-d_G(x)-d_G(y)> 2\Delta+2-2q\geq q$. Hence each such neighbor $z$ distributes to $x$ at least $\mu_G(xz)\frac{d_G(z)-q}{d_G(z)} \ge  \mu_G(xz)\cdot\frac{2\Delta+2-d_G(x)-d_G(y)-q}{2\Delta+2-d_G(x)-d_G(y)}$. Then,
\begin{align*}
    M'(x) & \geq M(x)+(\Delta-q+1) \cdot \frac{2\Delta+2-d_G(x)-d_G(y)-q}{2\Delta+2-d_G(x)-d_G(y)} \\ 
    & \ge d_G(x)+(\Delta-q+1)\cdot\frac{2\Delta+2-d_G(x)-2q}{2\Delta+2-d_G(x)-q}.
\end{align*}
Since $q \leq \frac{\sqrt{17}-3}{2}(\Delta+1)$ and $\Delta-q+1 < d_G(x) < q$, the function $$
f(t):=t+(\Delta-q+1)\cdot\frac{2\Delta+2-t-2q}{2\Delta+2-t-q}$$ is increasing. Hence,
\begin{align*}
M'(x) & > f(\Delta-q+1) \\ 
&=(\Delta-q+1)+(\Delta-q+1)\cdot\frac{\Delta-q+1}{\Delta+1} \\
& \geq q,    
\end{align*}
where the last inequality follows from
\(q \le \frac{\sqrt{17}-3}{2}(\Delta+1) < (2-\sqrt2)(\Delta+1)\).

\noindent \textbf{Subcase 2.2.} $d_G(y) > q$. That is, all neighbors of $x$ have degree greater than $q$. Note that since $d_G(x) < q <\frac{2\Delta+2}{3}$, we have $q<\Delta+1-\tfrac{d_G(x)}{2}<\Delta$.
We further consider two cases according to whether $$d_G(y) \ge \Delta+1-\frac{d_G(x)}{2}.$$ In either case, we show that
$$
M'(x)\ge 2\left(d_G(x)- \frac{q\,d_G(x)}{2\Delta+2-d_G(x)}\right).
$$

If $d_G(y)\geq \Delta + 1 - \frac{d_G(x)}{2}$, then each neighbor $z\in N_G(x)$ distributes to $x$ at least $\mu_G(xz) \frac{d_G(z)-q}{d_G(z)} \geq \mu_G(xz)\frac{d_G(y)-q}{d_G(y)}$. Thus
\begin{align*}
    M'(x) & \geq M(x)+d_G(x)\cdot \frac{d_G(y)-q}{d_G(y)} \\
    & \geq d_G(x)+d_G(x) \cdot \frac{2\Delta+2-d_G(x)-2q}{2\Delta+2-d_G(x)} \\ 
    & = 2\left(d_G(x)- \frac{q\,d_G(x)}{2\Delta+2-d_G(x)}\right).
\end{align*}

If $d_G(y)<\Delta+1-\frac{d_G(x)}{2}$, then by Corollary~\ref{cor:degree}, $x$ is incident with at least $\Delta-d_G(y)+1$ edges $f=xz$ such that $d_G(z)\geq 2\Delta+2-d_G(x)-d_G(y) > d_G(y)$. Such a neighbor $z$ distributes to $x$ at least $\mu_G(xz) \frac{d_G(z)-q}{d_G(z)} \ge \mu_G(xz) \cdot \frac{2\Delta+2-d_G(x)-d_G(y)-q}{2\Delta+2-d_G(x)-d_G(y)} > \mu_G(xz) \frac{d_G(y)-q}{d_G(y)}$. Therefore,
\begin{align*} 
   M'(x) &\geq d_G(x)+ (\Delta -d_G(y)+1) \cdot \frac{2\Delta+2-d_G(x)-d_G(y)-q}{2\Delta+2-d_G(x)-d_G(y)}+ (d_G(x)-\Delta+d_G(y)-1)\cdot \frac{d_G(y)-q}{d_G(y)} \\ 
    & =2d_G(x)-\left(\frac{\Delta-d_G(y)+1}{2\Delta+2-d_G(x)-d_G(y)}+\frac{d_G(x)-\Delta+d_G(y)-1}{d_G(y)}\right)\,q.
\end{align*}
Let $f(t) = \frac{\Delta-t+1}{2\Delta+2-d_G(x)-t}+\frac{d_G(x)-\Delta+t-1}{t}$, where $t=d_G(y)$ and $q < t < \Delta+1-\tfrac{d_G(x)}{2}$.
Then 
\begin{align*}
f'(t) &=\frac{d_G(x)-\Delta-1}{(2\Delta+2-d_G(x)-t)^2} + \frac{\Delta-d_G(x)+1}{t^2} \\
&= (\Delta-d_G(x)+1)\left(\frac{1}{t^2}-\frac{1}{(2\Delta+2-d_G(x)-t)^2}\right).   
\end{align*}
Since $t < \Delta+1-\frac{d_G(x)}{2}$, we have $2\Delta+2-d_G(x)-t>t$. Hence $f'(t) > 0$ and $f(t) < \tfrac{2d_G(x)}{2\Delta+2-d_G(x)}$. This implies that
\begin{align*}
    M'(x) & > 2d_G(x)- \left(\frac{2d_G(x)}{2\Delta+2-d_G(x)}\right)\,q = 2\left(d_G(x)-\frac{q\,d_G(x)}{2\Delta+2-d_G(x)}\right). 
\end{align*}

It remains to show that $M'(x) \ge 2\left(d_G(x)- \frac{q\,d_G(x)}{2\Delta+2-d_G(x)}\right)\ge q$.
Let $g(t)=t-\frac{qt}{2\Delta+2-t}$, where $t = d_G(x)$ and $\Delta-q+1 < t < q$. Then $g'(t)=1- \frac{(2\Delta+2)q}{(2\Delta+2-t)^2}$. Since $t<q\le \tfrac{\sqrt{17}-3}{2}(\Delta+1) <(3-\sqrt{5})(\Delta+1)$, we have $g'(t)>0$. Consequently, 
\begin{align*}
    M'(x) & > 2\left(\Delta-q+1-\frac{q(\Delta-q+1)}{2\Delta+2-(\Delta-q+1)}\right)\\ 
    & = \frac{2(\Delta-q+1)(\Delta+1)}{\Delta+1+q}\\
    & \ge q,
\end{align*}
where the last inequality follows from the assumption $q\leq \frac{\sqrt{17}-3}{2}(\Delta+1)$.

Therefore $M'(x) \ge q$ for every $x \in V(G)$. This proves that $\overline{d}(G) \ge \min\Big\{\tfrac{\Delta+\sqrt{2\Delta-1}}{2}, \tfrac{\sqrt{17}-3}{2}(\Delta+1)\Big\}$.

We now prove the second assertion. Let $A = \tfrac{2\mu\Delta+2\mu(2\mu-1)}{4\mu-1}$, $B=\frac{\sqrt{17}-3}{2}(\Delta+1)$, and $q = \min\{A,B\}$.
By hypothesis, either $$\mu \le \frac{1+\sqrt{2\Delta-1}}{4},$$ or $$\mu \le \frac 12 + \frac{\sqrt{2}}{4}\sqrt{(5-\sqrt{17})\Delta+(3-\sqrt{17})}.$$
In the former case, a routine calculation yields $A \le \Delta-(2\mu-1)^2$, while in the latter case, $B \le \Delta-(2\mu-1)^2$. Hence $q \le \Delta-(2\mu-1)^2$, or equivalently, $\mu \le \frac{\sqrt{\Delta-q}+1}{2}$.

We refine the argument used to prove the general bound. Observe that the proof of Case~2 above only uses the assumption $q\le \frac{\sqrt{17}-3}{2}(\Delta+1)$, which still holds since $q=\min\{A,B\}\le B$. Hence the same argument yields $M'(x)\ge q$ whenever $\Delta-q+1<d_G(x)<q$.
Thus we may assume that $d_G(x) \le \Delta-q+1$. We claim that
$$
M'(x) \ge \Delta-q+2\mu+\frac{\Delta-q}{2\mu-1}.
$$
To prove this bound, we refine the analysis in Case~1 above and consider the following two cases.

\medskip
\noindent \textbf{Case (i).} $d_G(x)\le 2\mu$. 
As in Case~1 above, Corollary~\ref{cor:degree} implies that for every $y\in N_G(x)$, $m_y\le d_G(y)+d_G(x)-\Delta-1$. Hence $y$ sends to $x$ at least
$\mu_G(xy)\frac{\Delta-q}{d_G(x)-1}$, and so
$$
M'(x)\ge M(x)+d_G(x)\cdot\frac{\Delta-q}{d_G(x)-1}
= \Delta-q+1+(d_G(x)-1)+\frac{\Delta-q}{d_G(x)-1}.
$$
Under the assumption $d_G(x) \le 2\mu \le \sqrt{\Delta-q}+1$, the function $d_G(x)-1+\tfrac{\Delta-q}{d_G(x)-1}$
is non-increasing in $d_G(x)$ on this interval. Hence $M'(x) \ge \Delta-q+2\mu+\tfrac{\Delta-q}{2\mu-1}$.

\noindent \textbf{Case (ii).} $2\mu<d_G(x)\le\Delta-q+1$. Let $y \in N_G(x)$ be chosen so that $\sigma'(x,y)$ is minimized, and let $Z$ be the set given by Lemma~\ref{lem: sigma}. For any $z \in N_G(x)$, $\sigma'(x, z) \ge q$, and every vertex of $K(x,z)$ has degree at least $q+1$, yielding $m_z \le d_G(z) - \sigma'(x,z)$. Thus
$$M'(x) \ge M(x)+\sum_{z \in N_G(x)}\frac{d_G(z)-q}{d_G(z)-\sigma'(x,z)} \cdot \mu_G(xz) \ge d_G(x)+\sum_{z\in N_G(x)}\frac{\Delta-q}{\Delta-\sigma'(x,z)}\cdot \mu_G(xz).$$
Moreover, by the choice of $y$, we have $\sigma'(x,z) \ge \sigma'(x, y)$ for all $z \in N_G(x)$, and for $z \in Z$, $\sigma'(x,z) \ge 2\Delta-d_G(x)-2\mu+2-\sigma'(x,y)$.

If $2\Delta-d_G(x)-2\mu+2-\sigma'(x,y) \ge \sigma'(x, y)$, that is, $\sigma'(x,y)\le \Delta-\mu+1-\tfrac{d_G(x)}{2}$, then
\begin{align*}
M'(x) &\ge d_G(x)+(\Delta-q)\left(\sum_{z\in Z}\frac{\mu_G(xz)}{\Delta-\sigma'(x,z)}+\sum_{w\in N_G(x)\setminus Z} \frac{\mu_G(xw)}{\Delta-\sigma'(x,w)}\right)\\
&\ge d_G(x)+(\Delta-q)\left(\frac{|E_G(x,Z)|}{\sigma'(x,y)-\Delta+d_G(x)+2\mu-2}+\frac{d_G(x)-|E_G(x,Z)|}{\Delta-\sigma'(x,y)}\right).
\end{align*}
By the assumption on $\sigma'(x,y)$ and the bound $|E_G(x, Z)| \ge \Delta-\sigma'(x,y)-\mu+1$, we obtain
$$M'(x)\ge d_G(x)+(\Delta-q)\left(\frac{\Delta-\sigma'(x,y)-\mu+1}{\sigma'(x,y)-\Delta+d_G(x)+2\mu-2}+\frac{d_G(x)-\Delta+\sigma'(x,y)+\mu-1}{\Delta-\sigma'(x,y)}\right).
$$
Let $h(\sigma') := \frac{\Delta-\sigma'-\mu+1}{\sigma'-\Delta+d_G(x)+2\mu-2}+\frac{d_G(x)-\Delta+\sigma'+\mu-1}{\Delta-\sigma'}$, where $\sigma'=\sigma'(x,y)$ and $\sigma' \le \Delta-\mu+1-\tfrac{d_G(x)}{2}$. A direct calculation gives
$$
h'(\sigma') = (d_G(x)+\mu-1)\left(\frac{1}{(\Delta-\sigma')^2}-\frac{1}{(\sigma'-\Delta+d_G(x)+2\mu-2)^2}\right) \le 0,
$$
so $h(\sigma')$ is non-increasing. Hence 
$h(\sigma')\ge \tfrac{2d_G(x)}{2\mu-2+d_G(x)}$, and
$$M'(x) \ge d_G(x)+(\Delta-q)\frac{2d_G(x)}{d_G(x)+2\mu-2}.$$

Otherwise, $\sigma'(x,y)> \Delta-\mu+1-\tfrac{d_G(x)}{2}$. Then, 
\begin{align*}
M'(x)
&\ge d_G(x)+d_G(x)\frac{\Delta-q}{\Delta-\sigma'(x,y)}\\
&> d_G(x) + (\Delta-q)\frac{2d_G(x)}{2\mu-2+d_G(x)}.
\end{align*}

Thus $$M'(x) \ge d_G(x)+(\Delta-q)\frac{2d_G(x)}{2\mu-2+d_G(x)}.$$
Since the function $t \mapsto t+(\Delta-q)\frac{2t}{t+2\mu-2}$
is increasing and $d_G(x) > 2\mu$, it follows that
$$
M'(x)> 2\mu+(\Delta-q)\frac{2\mu}{2\mu-1},
$$ as desired.

In either case, $M'(x) \ge \Delta-q+2\mu+\frac{\Delta-q}{2\mu-1} \ge q$, where the last inequality holds since $q \le \frac{2\mu\Delta+2\mu(2\mu-1)}{4\mu-1}$.
This completes the proof.
\qed

\section{Proof of Theorem~\ref{thm:avg-degree-2-8}}\label{sec:2-8}

Let $G$ be a $\Delta$-critical graph with maximum degree $\Delta \in \{2, 3, 4, 5, 6, 7, 8\}$. By Observation~\ref{obs:critical}, $\delta(G) \ge 2$. Set $q = (2\Delta+2)/3$.
If $\Delta = 2$, then $q = 2$ and $\overline{d}(G) \ge \delta(G) \ge 2$, so the result follows. We now consider $\Delta \ge 3$. We use the discharging method and show that after the discharging process every vertex has charge at least $q$. Assign to each vertex $x$ of $G$ an initial charge $M(x)=d_G(x)$. 
For each vertex $y$, let $E_y$ be the set of edges $e\in E_G(x,y)$ with $M(x) < q$, and set $m_y=|E_y|$.  

The discharging rules depend on the value of $\Delta$: one rule applies when $\Delta \ne 7$, while a modified version is used when $\Delta = 7$. First consider $\Delta \ne 7$. The following rule applies:
    \begin{itemize}
        \item \textbf{Rule.} Each vertex $y \in V(G)$ with $M(y) > q$ and $m_y \ge 1$ sends $\mu_G(xy)\frac{M(y)-q}{m_y}$ to each neighbor $x$ with $M(x)<q$. 
    \end{itemize}
Denote by $M'(x)$ the resulting charge on each vertex $x$. For any vertex $y$ with $M(y) \ge q$, we have $M'(y) \ge q$, with equality whenever $m_y \ne 0$. It remains to consider vertices with initial charge less than $q$.

\begin{claim}\label{claim:rechargeDelta+2}
If $xy$ is an edge of $G$ such that $d_G(x) \le d_G(y)$ and $d_G(x)+d_G(y) = \Delta + 2$, then $\mu_G(xy)=1$ and $M'(x)+M'(y) \ge 2q$. Moreover, if $d_G(y) \ge q$, then $M'(x) \ge q$.
\end{claim}
\begin{proof}
Since $\Delta > 2$, $d_G(x) + d_G(y) = \Delta + 2 < 2q$, and hence $d_G(x) < q < \Delta$, so $m_y \ne 0$.
Applying Corollary~\ref{cor:degree} to the vertices $x, y$, we obtain sets $Z_x\subseteq N_G(x)\setminus\{y\}$ and $Z_y\subseteq N_G(y)\setminus\{x\}$ such that 
$$|E_G(x, Z_x)| \ge \Delta-d_G(y)+1 \quad \text{and} \quad |E_G(y, Z_y)| \ge \Delta-d_G(x)+1.$$ 
Then, $|E_G(x, Z_x)|+|E_G(y, Z_y)| \ge 2\Delta+2-(d_G(x)+d_G(y))=\Delta$.
On the other hand, 
$$|E_G(x,N_G(x)\setminus \{y\})|+|E_G(y,N_G(y)\setminus\{x\})|=d_G(x)+d_G(y)-2\mu_G(xy) =\Delta+2-2\mu_G(xy).$$
It follows that $\mu_G(xy) = 1$ and $N_G(\{x,y\})\setminus\{x, y\} = Z_x \cup Z_y$.
Moreover, by Corollary~\ref{cor:degree}$(a)$ each vertex $z \in N_G(\{x,y\})\setminus\{x, y\}$ has degree $\Delta$.
If $d_G(y) > q$, then $y$ sends charge only to $x$. As $|E_G(x, Z_x)| \ge \Delta-d_G(y)+1 = d_G(x)-1 \ge 1$, $Z_x \ne \emptyset$. For any $z \in Z_x$, $m_z \ne 0$, and every neighbor of $z$ outside $\{x, y\}$ is a $\Delta$-vertex, since $d_G(x) < \Delta$. Hence $M'(z) = q$, and $z$ sends exactly $\Delta - q$ either to $x$ alone or to both $x$ and $y$. Consequently, $M'(x) + M'(y)=M(x)+M(y)+(\Delta-q) \ge 2\Delta + 2 - q = 2q$. If $d_G(y) \ge q$, then $M'(y) = q$, and thus $M'(x) \ge q$.
\end{proof}

Let $x \in V(G)$. If $d_G(x) = 2$, then every $y \in N_G(x)$ satisfies $d_G(y) = \Delta > q$. By Claim~\ref{claim:rechargeDelta+2}, we get $M'(x) \ge q$. Thus $2$-vertices require no further consideration. When $\Delta = 3$, we have $q = 8/3 < 3$, so $M'(x) \ge q$ for every vertex $x$, completing the case $\Delta = 3$.

Assume $\Delta \in \{4,5,6,8\}$. The above argument still applies, but additional cases must be considered. We distinguish cases according to the value of $\Delta$. Let $x\in V(G)$ be a vertex of the type considered in the corresponding case, and let $y$ be a neighbor of $x$ satisfying $d_G(y)=\min \{d_G(z) \text{: } z\in N_G(x)\}$. We have $d_G(y)\ge \Delta+2-d_G(x)$.

\medskip
\noindent
\textbf{Case 1:} $\Delta = 4$. In this case, $q=\tfrac{10}{3}$. Since $2$-vertices have already been handled, it suffices to consider $3$-vertices. Let $x$ be a $3$-vertex. Then $d_G(y) \ge 3$. Suppose $d_G(y) = 3$. By Corollary~\ref{cor:degree}, $|E_G(x,Z_x)|,\ |E_G(y, Z_y)| \ge 2$, and hence $\mu_G(xy) = 1$ and $N_G(\{x, y\})\setminus\{x, y\} = Z_x \cup Z_y$. Also, every vertex $z \in N(\{x, y\})\setminus\{x,y\}$ is a $4$-vertex, and every neighbor of $z$ other than $x$ and $y$ is also a $4$-vertex as $d_G(x), d_G(y) < \Delta$. Hence each such $z$ sends charge only to vertices in $\{x,y\}$, so $m_z = \mu_G(xz)+\mu_G(yz) \le 4$. Therefore $$M'(x) \ge M(x)+ (d_G(x)-\mu_G(xy))\cdot \frac{\Delta-q}{4} = 3+2\cdot\frac{4-\frac{10}{3}}{4} = \frac{10}{3}=q.$$ 
Moreover, $M'(x)=q$ if and only if equality holds in the above inequalities, that is, if and only if $x$ has a unique neighbor $z\ne y$ with
$\mu_G(xz)=\mu_G(yz)=2$. Equivalently, this occurs exactly when $G$ is a triangle on vertices $x,y,z$ with $\mu_G(xz)=\mu_G(yz)=2$ and $\mu_G(xy)=1$. 
Now suppose that $x$ only has $4$-neighbors. Then
\[
M'(x) \ge M(x)+\frac{\Delta-q}{4} \cdot d_G(x) = 3+\frac{4-\frac{10}{3}}{4}\cdot 3 = \frac{7}{2} > q.
\]
Therefore $M'(x) \ge q$ for every $x \in V(G)$, and the result follows for $\Delta = 4$.

For the remaining cases with $\Delta \in \{5, 6, 8\}$, we claim that if $d_G(x) = 3$, then $M'(x) \ge q$. Indeed, since $\Delta \ge 5$, $d_G(y) \ge \Delta +2 -d_G(x) = \Delta-1 \ge q$. Hence $M'(x) \ge q$ whenever $d_G(y) = \Delta - 1$ by Claim~\ref{claim:rechargeDelta+2}. We may assume that $d_G(y) = \Delta$, that is, each neighbor of $x$ is a $\Delta$-vertex. Applying Corollary~\ref{cor:degree} to each $\Delta$-neighbor of $x$, we obtain that every such neighbor is incident with at least $\Delta-2$ edges whose other endvertex has degree at least $\Delta-1 \ge q$. Thus $m_y \le 2$ for any $y \in N_G(x)$ and then
\[
M'(x) \ge M(x)+\frac{\Delta-q}{2} \cdot d_G(x) = 3 + \frac{\Delta-2}{2} =  \frac{\Delta+4}{2} \ge q,
\]
where the last inequality holds since $\Delta \le 8$.
Thus $3$-vertices require no further consideration. 

\medskip
\noindent
\textbf{Case 2:} $\Delta = 5$. Then $q = 4$, and $M'(x) \ge q$ for all $x \in V(G)$ follows from the preceding argument.

\medskip
\noindent
\textbf{Case 3:} $\Delta = 6$. Then, $q= \tfrac{14}{3} < 5$. In this case, we consider $4$-vertices only. 
    
Suppose $d_G(x) = 4$. Then $d_G(y) \ge 4$. If $d_G(y) = 4$, by Corollary~\ref{cor:degree}, $|E_G(x, Z_x)|,|E_G(y, Z_y)| \ge 3$, and thus $\mu_G(xy) = 1$ and $N_G(\{x, y\})\setminus\{x, y\} = Z_x\cup Z_y$. Moreover, every vertex in $N_G(\{x,y\})\setminus\{x,y\}$ is a $6$-vertex, and every neighbor of such a vertex is either $x$, $y$, or a $6$-vertex. Thus for each $z\in N_G(\{x,y\})\setminus\{x,y\}$, $m_z = \mu_G(xz)+\mu_G(yz) \le 6$. Therefore, $M'(x) \ge 4+3\cdot\tfrac{6-14/3}{6} = \tfrac{14}{3}$, where equality holds if and only if $G$ is a triangle on vertices $x, y, z$ with $\mu_G(xz) = \mu_G(yz) = 3$ and $\mu_G(xy) = 1$.
We may assume that $d_G(y) \ge 5$.
Then, we partition $N_G(x)$ into two sets. Let $V_5$ and $V_6$ denote the sets of $5$-neighbors and $6$-neighbors of $x$, respectively. Note that at most one of the sets may be empty.
Let $y' \in V_5$. Applying Corollary~\ref{cor:degree} to $x, y'$, we obtain that $y'$ is incident with at least three edges whose other endvertex has degree at least $5$. Hence $m_{y'} \le 2$. So each vertex $y' \in V_5$ distributes at least $\mu_G(xy') \cdot \tfrac{5-q}{2} =\tfrac 16 \mu_G(xy')$ to $x$. Let $y' \in V_6$. Then each neighbor $y'$ distributes at least $\mu_G(xy') \cdot \frac{6-q}{6}=\tfrac 29 \mu_G(xy')$ to $x$. Thus
\[
M'(x) \ge M(x) + \frac 16 \sum_{y'\in V_5} \mu_G(xy') + \frac 29 \sum_{y'\in V_6} \mu_G(xy') \\
\ge M(x) + \frac{1}{6} \cdot d_G(x) = 4 + \frac{2}{3} = \frac{14}{3} = q.
\]
Therefore $M'(x) \ge q$ for every $x \in V(G)$.

\medskip
\noindent
\textbf{Case 4:} $\Delta = 8$. Then $q = 6$, so in this case, we consider $4$- and $5$-vertices.

Suppose $d_G(x) = 4$. By Claim~\ref{claim:rechargeDelta+2}, we may assume $d_G(y) \ge 7$. For any vertex $z \in N_G(x)$, by Corollary~\ref{cor:degree}, $z$ is incident with at least $5$ edges whose other endvertex has degree at least $q$. Consequently, each $z$ distributes at least $\frac{1}{2}\mu_G(xz)$ to $x$ if $M(z) = 7$, and at least $\frac{2}{3}\mu_G(xz)$ if $M(z) = 8$. It follows that $M'(x) \ge M(x) + \tfrac 12 \cdot d_G(x) = 6$. 

Suppose $d_G(x) = 5$. Then $d_G(y) \ge 5$. If $d_G(y) = 5$, then, as in Case~3, we have $\mu_G(xy) = 1$, every vertex in $N_G(\{x,y\})\setminus\{x,y\}$ is an $8$-vertex, and each neighbor of such a vertex is either $x$, $y$, or an $8$-vertex. Hence for each $z \in N_G(\{x,y\})\setminus\{x,y\}$, $m_z = \mu_G(zx)+\mu_G(zy) \le 8$. Therefore, $M'(x) \ge 5+4\cdot \tfrac{8-6}{8}=6$, with equality if and only if $G$ is a triangle on vertices $x,y,z$ with $\mu_G(xz)=\mu_G(yz)=4$ and $\mu_G(xy)=1$. We now assume $d_G(y) = 6$. By Corollary~\ref{cor:degree}, there exists a nonempty set $Z_x \subseteq N_G(x)\setminus\{y\}$ with $|E_G(x, Z_x)| \ge 3$ such that each $z \in Z_x$ is incident with at least $7$ edges whose other endvertex is either $x$ or has degree at least $6$. 
We claim that each vertex in $Z_x$ sends at least $1$ to $x$. Let $z \in Z_x$. If $d_G(z) = 7$, then $x$ is the only neighbor of $z$ with degree less than $q$, and hence $z$ sends exactly $M(z)-q = 1$ to $x$. If $d_G(z) = 8$, then $m_z \le \mu_G(xz)+1$, and thus $z$ sends at least $(M(z)-q)\tfrac{\mu_G(xz)}{\mu_G(xz)+1}$ to $x$. Since $\mu_G(xz) \ge 1$, we have $\tfrac{\mu_G(xz)}{\mu_G(xz)+1} \ge \tfrac12$. It follows that $z$ sends at least $1$ to $x$. As $|Z_x| \ge 1$, we obtain $M'(x) \ge 6$. We now assume $d_G(y) \ge 7$. Let $z \in N_G(x)$. If $M(z) = 7$, then $z$ is incident with at least $4$ edges whose other endvertex has degree at least $6$, so $m_{z} \le 3$. Thus $z$ sends at least $\frac{1}{3}\mu_G(xz)$ to $x$. If $M(z) = 8$, then $z$ sends at least $\frac{1}{4}\mu_G(xz)$ to $x$. Hence $M'(x) \ge M(x) + \frac{1}{4} \cdot d_G(x) = \tfrac{25}{4} > 6$. Therefore $M'(x) \ge q$ for all $x \in V(G)$, and this case is complete.

We now turn to the case $\Delta = 7$ and $q = \tfrac{16}{3} < 6$, where we modify the rule as follows:
\begin{itemize}
\item \textbf{Rule 1.} Each $7$-vertex in $G$ distributes $\frac{1}{3}$ to each of its $5$-neighbors. 
\end{itemize}

Let $M_1(x)$ denote the charge of $x$ after Rule~1. Then $M_1(x) = d_G(x)$ if $d_G(x) \notin \{5, 7\}$. For each vertex $y$, let $E'_y$ be the set of edges $e\in E_G(x,y)$ with $M_1(x) < q$, and let $m'_y=|E'_y|$.
\begin{itemize}
\item \textbf{Rule 2.} Each vertex $y \in V(G)$ with $\min\{M(y), M_1(y)\} > q$ and $m'_y \ge 1$ sends $\mu_G(xy)\frac{M_1(y)-q}{m'_y}$ to each neighbor $x$ with $M_1(x)<q$.
\end{itemize}

Let $M'(x)$ denote the charge of $x$ after Rule~2. Clearly, $M'(x) \ge q$ if $M(x) = 5$ and $x$ has a $7$-neighbor, or if $M(x) = 6$.
We claim that if $M(x) = 7$, then $M_1(x) \ge q$, and hence $M'(x) \ge q$. It suffices to show that $x$ has at most five $5$-neighbors. Indeed, let $y$ be a $5$-neighbor of $x$. Applying Corollary~\ref{cor:degree} to $x$ and $y$, we obtain a set $Z_x \subseteq N_G(x)\setminus \{y\}$ such that $|E_G(x, Z_x)| \ge 3$ and $\sum_{z\in Z_x}d_G(z) \ge \Delta|Z_x| - 3$. This implies that at most one vertex in $Z_x$ has degree $5$. Hence $x$ has at most $\Delta - 3 + 1 = 5$ neighbors of degree $5$. 
Moreover, if $M(x) \ge 6$ and $m'_x \ne 0$, then $M'(x) = q$.

We claim that Claim~\ref{claim:rechargeDelta+2} still holds under the modified rules. 
The argument is similar to that for $\Delta \ne 7$. The differences arise when $d_G(x) = 2$ and $d_G(y) = 7$, or when $d_G(x) = 4$ and $d_G(y) = 5$. In both cases, by the same argument as in the proof of
Claim~\ref{claim:rechargeDelta+2} (which does not depend on the
discharging rule), we have $\mu_G(xy) = 1$, and every vertex in $N_G(\{x, y\})\setminus\{x, y\}$ is a $\Delta$-vertex. Moreover, every neighbor of such a vertex, except $x$ and $y$, has degree at least $6$. Let $z \in N_G(x) \setminus \{y\}$. In the former case, after Rule~1, we have $M_1(x) = 2$, $M_1(y) = 7$, and since $z$ has no $5$-neighbor, $M_1(z) = 7$. As both $y$ and $z$ have only one neighbor, namely $x$, of degree less than $q$, each of $y$ and $z$ sends $7 - q = 5/3$ to $x$ in Rule~2. Consequently, $M'(x) = 2 + \tfrac{10}{3}= \tfrac{16}{3}$ and $M'(y) = \tfrac{16}{3}$. In the latter case, $M_1(x) = 4$ and $z$ has at most one $5$-neighbor, namely $y$. If $yz \in E(G)$, then $M_1(z) = \tfrac{20}{3}$ and $M_1(y) \ge \tfrac{16}{3}$. This implies that $y$ neither sends nor receives charge in Rule~2, so $M'(y) \ge \tfrac{16}{3}$. Also, $z$ sends $M_1(z)-q = \tfrac{4}{3}$ to $x$ in Rule~2, so $M'(x) \ge \tfrac{16}{3}$, and hence $M'(x) + M'(y) \ge 2q$. If $yz \notin E(G)$, then $z$ has only one neighbor of degree less than $q$, namely $x$. Hence $z$ sends $\Delta-q$ to $x$ in Rule~2, so $M'(x)\ge d_G(x)+\Delta-q$. Since $M'(y) \ge M_1(y) \ge d_G(y)$, $M'(x)+M'(y) \ge 2q$. The claim follows. 

Therefore it remains to consider $3$- and $4$-vertices, and $5$-vertices without $7$-neighbors. Let $x$ be such a vertex, and let $y \in N_G(x)$ be a neighbor of minimum degree. Then, $d_G(y) \ge 9-d_G(x)$.

Suppose $d_G(x) = 3$ and $d_G(y) \ge 6$. By Claim~\ref{claim:rechargeDelta+2}, $M'(x) \ge q$ if $d_G(y) = 6$. We may assume all neighbors of $x$ are $7$-vertices. By Corollary~\ref{cor:degree}, each is incident with at least five edges whose other endvertex has degree at least $6$. Thus, for any $z \in N_G(x)$, we have either $M_1(z) = 20/3$ and $m'_z = 1$, or $M_1(z) = 7$ and $m'_z \le 2$. Each $z$ sends either exactly $\tfrac{4}{3}$, in which case $\mu_G(xz) = 1$, or at least $\tfrac{5}{6}\mu_G(xz)$ to $x$ in Rule~2. Thus $M'(x) \ge M_1(x)+\tfrac{5}{6}\cdot d_G(x) = 3+\tfrac{5}{6}\cdot 3 = \tfrac{11}{2} > q$. 

Suppose $d_G(x) = 4$ and $d_G(y) \ge 5$. If $d_G(y) = 5$, then, as in the proof of Claim~\ref{claim:rechargeDelta+2}, $M'(x), M'(y) \ge q$ whenever $x$ and $y$ have a common $\Delta$-neighbor. We may assume $x$ and $y$ have no common neighbor. Then there exist distinct $\Delta$-vertices $z_1, z_2$ such that $z_1\in N_G(x)\setminus\{y\}$ and $z_2\in N_G(y)\setminus\{x\}$. In Rule~1, $z_2$ sends $1/3$ to $y$, and in Rule~2, $z_1$ sends $\Delta-q = 5/3$ to $x$. Thus $M'(x) \ge 17/3$ and $M'(y) \ge 16/3$. We next consider $d_G(y) \ge 6$. It suffices to show that each $6$- or $7$-neighbor $z$ of $x$ sends at least $\tfrac{1}{3}\mu_G(xz)$ to $x$ in Rule~2, which yields 
$M'(x) \ge M_1(x) + \tfrac{1}{3}\cdot d_G(x) = \tfrac{16}{3}$. If $d_G(z) = 6$, $M_1(z) = 6$ and it is incident with at least four edges whose other endvertex has degree at least $6$, so $z$ sends $x$ at least $\tfrac{1}{3}\mu_G(xz)$ in Rule~2. If $d_G(z) = 7$, 
then by Corollary~\ref{cor:degree}, there is a set $Z_z \subseteq N_G(z)\setminus\{x\}$ with $|E_G(z, Z_z)| \ge 4$ such that 
$$
\sum_{u\in Z_z}d_G(u) \ge \Delta|Z_z|-2.
$$ 
This implies that $Z_z$ contains at most one vertex of degree less than $q$, and if so, it must be a $5$-vertex. Thus $z$ has at most four neighbors of degree less than $q$, of which at most three are $5$-vertices. We consider the charge sent from $z$ to $x$ in Rule~2 according to the number of $5$-neighbors of $z$, as summarized in the following table.
\renewcommand{\arraystretch}{1.3}
\begin{table}[H]
\centering
\begin{tabular}{|c|c|c|c|}
\hline
number of $5$-neighbors & $M_1(z)$ & upper bound on $m'_z$ & lower bound on charge to $x$ \\ \hline
0 & 7 & 3 & $\tfrac{5}{9}\mu_G(xz)$ \\ \hline
1 & $20/3$ & 3 & $\tfrac{4}{9}\mu_G(xz)$ \\ \hline
2 & $19/3$ & 2 & $\tfrac{1}{2}\mu_G(xz)$ \\ \hline
3 & 6 & 1 & $\tfrac{2}{3}\mu_G(xz)$ \\ \hline
\end{tabular}
\caption{Lower bounds on the charge sent from $z$ to $x$}
\end{table}

\noindent Therefore, in all cases, the charge sent from $z$ to $x$ is at least $\tfrac{1}{3}\mu_G(xz)$ in Rule~2.

Finally, suppose $d_G(x) = 5$ and $x$ has no $7$-neighbor. Then, $M_1(x) = M(x) = 5$ and $d_G(y) \ge 4$. The case where $d_G(y) = 4$ has already been considered, so we may assume that $d_G(y) \ge 5$. Let $V_5 = \{z \in N_G(x): d_G(z) = 5\}$ and $V_6 = \{z \in N_G(x): d_G(z) = 6\}$. Then $N_G(x) = V_5 \cup V_6$. If $V_5 \neq \emptyset$, then $d_G(y)= 5$ and by applying Corollary~\ref{cor:degree} to $x$ and $y$, we obtain a set $Z_x \subseteq N_G(x)\setminus\{y\}$ with $|E_G(x, Z_x)| \ge 3$ such that 
$$
\sum_{z\in Z_x}d_G(z) \ge \Delta|Z_x|-1.
$$
Since $x$ has no $7$-neighbor, $Z_x$ consists of a single vertex $z$ with $d_G(z) = 6$ and $\mu_G(xz) \ge 3$, which implies that $z$ sends at least $\tfrac12 (6-q) = \tfrac{1}{3}$ to $x$. Thus $M'(x) \ge M_1(x)+\frac{1}{3} = \tfrac{16}{3}$. If $V_5 = \emptyset$, then every neighbor $z$ of $x$ is a $6$-vertex and sends at least $\tfrac{6-q}{6}\cdot\mu_G(xz) = \tfrac{1}{9}\mu_G(xz)$ to $x$. Hence $M'(x) \ge M_1(x)+\tfrac{1}{9}\cdot d_G(x) = \tfrac{50}{9}$. In either case, $M'(x) \ge q$.

Therefore $M'(x) \ge q$ for any $x \in V(G)$. This completes the proof.\qed

\section{Generalized Woodall's Lemma to Multigraphs}\label{sec:prooflemma}
In this section, we first present several results that will be used in the proof of Lemma~\ref{lem: sigma}. Let $G$ be a $\Delta$-critical graph with multiplicity $\mu$, let $e=xy \in E_G(x,y)$, and let $\varphi \in \mathcal{C}^{\Delta}(G-e)$. Let $\alpha, \beta \in [\Delta]$. For a vertex $u\in V(G)$, let $P_u(\alpha, \beta, \varphi)$ denote the unique $(\alpha, \beta)$-chain that contains the vertex $u$. The coloring $\varphi' = \varphi/P_u(\alpha,\beta,\varphi)$ obtained from $\varphi$ by swapping the colors $\alpha$ and $\beta$ along the path $P_u(\alpha,\beta,\varphi)$ is also a proper coloring of $G-e$. Such a switching operation is called a \emph{Kempe change}. Moreover, if $\alpha \in \varphi(u)$, denote by $u_{\alpha,\varphi}$ the unique neighbor of $u$ such that $\varphi(f)=\alpha$ for some $f \in E_G(u,u_{\alpha,\varphi})$. When the coloring is clear, we simply write $u_\alpha$. Recall that $K_G(x,y)$ and $K_G(y, x)$ denote the Kierstead sets of $(x, y)$ and $(y, x)$, respectively, and that $Z_y(\varphi) \subseteq K_G(x, y)$ and $Z_x(\varphi) \subseteq K_G(y, x)$, where $Z_y(\varphi)$ and $Z_x(\varphi)$ are defined in Section~\ref{sec:pre}.
From the definition of $Z_y(\varphi)$ and $Z_x(\varphi)$, we may write
$$
Z_y(\varphi) = \{y_{\alpha,\varphi} : \alpha \in \overline{\varphi}(x)\}
\quad \text{and} \quad
Z_x(\varphi) = \{x_{\alpha,\varphi} : \alpha \in \overline{\varphi}(y)\}.
$$

\begin{lemma}\label{prop: 25}
Let $e \in E_G(x, y)$ be an edge in a $\Delta$-critical graph $G$ and $\varphi \in \mathcal{C}^{\Delta}(G - e)$. 
For every $\alpha \in \overline{\varphi}(x)$ and every $z \in Z_x(\varphi)$, $z_{\alpha,\varphi}$ is well-defined and $z_{\alpha,\varphi} \in K(x,z)$.

\end{lemma}

\begin{proof}
Since $z \in Z_x(\varphi)$, there is a color $\beta \in \overline{\varphi}(y)$ and an edge $e' \in E_G(x, z)$ such that $\varphi(e') = \beta$. Clearly, $\beta \ne \alpha$. By Lemma~\ref{lem:multi_fan}$(a)$, $\{x, y, z\}$ is $\varphi$-elementary, so $z_{\alpha,\varphi}$ exists and $z_{\alpha,\varphi} \ne x$. Now we uncolor $e'$ and color $e$ with $\beta$. It is easy to check that the resulting coloring $\varphi'$ is proper, and $z_{\alpha,\varphi} \in Z_z(\varphi') \subseteq K(x, z)$.
\end{proof}

\begin{lemma}\label{prop: 26}
Let $G$ be a $\Delta$-critical graph, let $e \in E_G(x,y)$, and let $\varphi \in \mathcal{C}^{\Delta}(G - e)$. Then the following statements hold:
\begin{enumerate}
\item[(a)] $$|\overline{\varphi}(x) \cup \overline{\varphi}(y)| = 2\Delta + 2 - d_G(x) - d_G(y),$$ and $$\mu_G(xy)-1 \le |\varphi(x) \cap \varphi(y)| = d_G(x) + d_G(y) - \Delta - 2.$$
\item[(b)] If $u \in N_G(\{x,y\}) \setminus \{x,y\}$ satisfies $d_G(x) + d_G(y) + d_G(u) < 2\Delta + 2$, then $$|\overline{\varphi}(u) \cap (\overline{\varphi}(x) \cup \overline{\varphi}(y))| \ge |\varphi(x) \cap \varphi(y) \cap \varphi(u)| + 1 \ge 1.$$
\end{enumerate}
\end{lemma}
The proof is similar to that of \cite[Proposition~4.26]{SSTF2012GraphEC}.

\begin{lemma}\label{prop: 27}
Let $G$ be a $\Delta$-critical graph and let $e \in E_G(x,y)$. 
Then there are at least $\Delta - \sigma'(x,y) - \mu_G(xy) + 1$ edges $f = xz$ with $z \in N_G(x) \setminus \{y\}$ such that there exists a coloring $\varphi \in \mathcal{C}^{\Delta}(G - e)$ with $\varphi(f) \in \overline{\varphi}(y)$.
\end{lemma}
\begin{proof}
Let $\varphi \in \mathcal{C}^{\Delta}(G-e)$.
Let $\Gamma$ be the set of all colors $\gamma \in \varphi(x)$ such that either $\gamma \in \overline{\varphi}(y)$, or $\gamma\in \varphi(y)$ and $y_{\gamma,\varphi} \notin K(x,y)\cup \{x\}$. By definition, $\Gamma$ contains none of the colors on the remaining $\mu_G(xy)-1$ edges between $x$ and $y$. Moreover, since $Z_y(\varphi) \subseteq K(x, y)$, each color in $\overline{\varphi}(x)$ appears on an edge from $y$ to $K(x,y)$. Hence the number of colors $\gamma \in \varphi(x) \cap \varphi(y)$ such that $y_{\gamma,\varphi} \in K(x,y)$ is $\sigma'(x,y) - |\overline{\varphi}(x)|$. Therefore
\[
|\Gamma|=|\varphi(x)|-\left(\sigma'(x,y)-|\overline{\varphi}(x)|\right)-\left(\mu_G(xy)-1\right)=\Delta-\sigma'(x,y)-\mu_G(xy)+1.
\]
Let $\gamma \in \Gamma$, and let $z = x_{\gamma,\varphi}$. Let $f = xz \in E_G(x,z)$ be the edge with $\varphi(f)=\gamma$. We show that there exists a coloring $\varphi' \in \mathcal{C}^{\Delta}(G - e)$ such that $\varphi'(f) \in \overline{\varphi'}(y)$.

If $\gamma \in \overline{\varphi}(y)$, then we are done. Hence we may assume $\gamma \notin \overline{\varphi}(y)$. Then, the vertex $u = y_{\gamma,\varphi}$ exists and $u\notin K(x,y)\cup \{x\}$. Let $f'\in E_G(y, u)$ be the edge colored by $\gamma$. Clearly, $u \ne z$ and $d_G(x) + d_G(y) + d_G(u) < 2\Delta + 2$. By Lemma~\ref{prop: 26}$(b)$, $\overline{\varphi}(u) \cap (\overline{\varphi}(x) \cup \overline{\varphi}(y)) \neq \emptyset$. If there is a color $\alpha \in \overline{\varphi}(u) \cap \overline{\varphi}(y)$, then recolor the edge $f'$ with $\alpha$. The resulting coloring $\varphi'$ is proper and satisfies $\varphi'(f) \in \overline{\varphi'}(y)$.

Otherwise, there is a color $\alpha \in \overline{\varphi}(u) \cap \overline{\varphi}(x)$. Let $\beta \in \overline{\varphi}(y)$. Clearly, $\alpha, \beta, \gamma$ are distinct. 
Lemma~\ref{lem:multi_fan}$(b)$ implies that there is a path $P = P_x(\alpha, \beta, \varphi)$ joining $x$ and $y$. As $\alpha \notin \varphi(u)$, we know $u \notin V(P)$. Let $\varphi_1 = \varphi/P$. Then, $\alpha \in \overline{\varphi_1}(y) \cap \overline{\varphi_1}(u)$ and $\varphi_1(f) = \varphi_1(f') = \gamma$. We now recolor $f'$ by $\alpha$. The resulting coloring $\varphi'$ is proper and satisfies $\varphi'(f) \in \overline{\varphi'}(y)$.
\end{proof}

We have the following corollary from Lemmas~\ref{lem:degree}$(a)$ and~\ref{prop: 27}.

\begin{corollary}\label{cor: Z}
Let $G$ be a $\Delta$-critical graph, and let $e \in E_G(x, y)$. Let $Z$ be the set of all vertices $z \in N_G(x) \setminus \{y\}$ for which there exists a coloring $\varphi \in \mathcal{C}^{\Delta}(G - e)$ such that some edge $f \in E_G(x,z)$ satisfies $\varphi(f) \in \overline{\varphi}(y)$. Then $Z \subseteq K_G(y,x)$ and $|E_G(x,Z)| \ge \Delta - \sigma'(x,y) - \mu_G(xy)+1 \ge \Delta-d_G(y)+1$.
\end{corollary}

Observe that $Z = \bigcup_{\varphi \in \mathcal{C}^{\Delta}(G-e)} Z_x(\varphi)$.
Now we show Lemma~\ref{lem: sigma}.

\noindent \textbf{Proof of Lemma~\ref{lem: sigma}.}

Let $Z$ be the set defined in Corollary~\ref{cor: Z}.

It remains to show that $\sigma'(x,y) + \sigma'(x,z) \ge 2\Delta - d_G(x)-2\mu+2$ for all vertices $z \in Z$.

\begin{claim}\label{claim:1}
Let $e \in E_G(x, y)$ and let $\varphi \in \mathcal{C}^{\Delta}(G - e)$. Let $z = x_{\alpha, \varphi}$, where $\alpha\in \overline{\varphi}(y)$, and let $\beta \in \varphi(x) \setminus \{\alpha\}$. Then the following statements hold:
\begin{enumerate}

\item[(a)] If $\beta \in \overline{\varphi}(y)$, then $z_{\beta}$ is well-defined and either $z_{\beta} = x$ or  $z_{\beta} \in K(x,z)$.
\item[(b)] If $\beta \in \overline{\varphi}(z)$, then $y_{\beta}$ is well-defined and either $y_{\beta} = x$ or $y_{\beta} \in K(x,y)$.
\end{enumerate}
\end{claim}
\begin{proof}
Let $f \in E_G(x,z)$ satisfy $\varphi(f) = \alpha$.

Suppose that $\beta \in \overline{\varphi}(y)$. Since $\{x, y, z\}$ is $\varphi$-elementary, $z_\beta$ is well-defined. Clearly, $z_\beta \ne y$. Let $f' \in E_G(z,z_\beta)$ be such that $\varphi(f') = \beta$. Assume $z_\beta \ne x$. Then $(y, e, x, f, z, f', z_\beta)$ is a Kierstead path with respect to $e$ and $\varphi$. By Lemma~\ref{lem:ShortK_Path}(c), $d_G(x) + d_G(z) + d_G(z_\beta) \ge 2\Delta + 2$, so $z_\beta \in K(x,z)$.

Suppose that $\beta \in \overline{\varphi}(z)$. Similarly, since $\{x, y, z\}$ is $\varphi$-elementary, $y_\beta$ is well-defined and $z \notin \{x_\beta, y_\beta\}$. Let $\varphi'$ be the coloring obtained from $\varphi$ by uncoloring $f$ and coloring $e$ with $\alpha$. Then $\varphi' \in \mathcal{C}^{\Delta}(G-f)$ satisfies $\alpha, \beta \in \overline{\varphi'}(z)$, $\beta \in \varphi'(x)$ and $\varphi'(e) = \alpha$. Hence we can apply statement $(a)$ to the coloring $\varphi'$ with the roles of $z$ and $y$ exchanged. Thus either $y_\beta = x$ or $y_\beta \in K(x, y)$.
\end{proof}

\begin{claim}\label{claim:2}
    Let $e =xy\in E_G(x, y)$ and let $\varphi \in \mathcal{C}^{\Delta}(G - e)$. Let $z = x_{\alpha,\varphi}$, where $\alpha \in \overline{\varphi}(y)$, and let $\beta \in \varphi(x)\cap \varphi(y)\cap \varphi(z)$. Then the following statements hold:
\begin{enumerate}
\item[(a)] If $z_\beta \notin K(x,z)$ and $x \notin \{y_\beta, z_\beta\}$, then $z_\beta \neq y$ and, moreover, $y_\beta \in K(x,y)$ or there is a color $\gamma$ such that $\gamma \in \overline{\varphi}(x)\cap \overline{\varphi}(z_\beta)\cap \varphi(z)\cap \varphi(y_\beta)$.
\item[(b)] If $y_\beta \notin K(x,y)$ and $x \notin \{y_\beta, z_\beta\}$, then $y_\beta \neq z$ and, moreover, $z_\beta \in K(x,z)$ or there is a color $\gamma'$ such that $\gamma' \in \overline{\varphi}(x)\cap \overline{\varphi}(y_\beta)\cap \varphi(y)\cap \varphi(z_\beta)$.
\end{enumerate}

\end{claim} 
\begin{proof}
For the proof of (a), assume that $z_\beta \notin K(x,z)$ and $x \notin \{y_\beta, z_\beta\}$. Moreover, $x_\beta, y_\beta, z_\beta$ are pairwise distinct. If $z_\beta = y$, then Lemma~\ref{lem:multi_fan}$(a)$ implies that $d_G(x) + d_G(z) + d_G(z_\beta) \ge 2\Delta + 2$. Thus $z_\beta \in K(x,z)$, a contradiction. Hence $z_\beta \notin \{x, y\}$, and $y_\beta \notin \{x, z\}$.

Since $z_\beta \notin K(x,z)$, $d_G(x) + d_G(z) + d_G(z_\beta) < 2\Delta + 2$. By Lemma~\ref{lem:multi_fan}$(a)$, $\overline{\varphi}(x) \cap \overline{\varphi}(z) = \emptyset$. We claim that $\overline{\varphi}(z_\beta) \cap (\overline{\varphi}(x) \cup \overline{\varphi}(z)) \neq \emptyset$, for otherwise, as $\beta \notin \overline{\varphi}(x) \cup \overline{\varphi}(z) \cup \overline{\varphi}(z_\beta)$, we would have 
\begin{align*}
 \Delta - 1 & \ge |\overline{\varphi}(x) \cup \overline{\varphi}(z) \cup \overline{\varphi}(z_\beta)| \\ &   = |\overline{\varphi}(x)| + |\overline{\varphi}(z)| + |\overline{\varphi}(z_\beta)|  \\ & = 3\Delta + 1 - d_G(x) - d_G(z) - d_G(z_\beta) \\ & > \Delta - 1,
\end{align*}
 a contradiction.

If there is a color $\gamma \in \overline{\varphi}(z) \cap \overline{\varphi}(z_\beta)$, then recolor the edge $zz_\beta$ with $\gamma$, color the edge $e$ with $\alpha$, and uncolor the edge $xz$. This results in a coloring $\varphi' \in \mathcal{C}^{\Delta}(G - xz)$ such that $\alpha, \beta \in \overline{\varphi'}(z)$, $\varphi'(xy) = \alpha$, and $\varphi'(yy_\beta) = \beta$. Then $(z, zx, x, xy, y, yy_\beta, y_\beta)$ is a Kierstead path with respect to $xz$ and $\varphi'$. By Lemma~\ref{lem:ShortK_Path}$(c)$, we have $d_G(x) + d_G(y) + d_G(y_\beta) \ge 2\Delta + 2$. Therefore, $y_\beta \in K(x,y)$.

Now we may assume that $\overline{\varphi}(z) \cap \overline{\varphi}(z_\beta) = \emptyset$. Since $\overline{\varphi}(z_\beta) \cap (\overline{\varphi}(x) \cup \overline{\varphi}(z)) \neq \emptyset$, this implies that there is a color $\gamma \in \overline{\varphi}(x) \cap \overline{\varphi}(z_\beta)$. Then we have $\gamma \in \varphi(z)\setminus\{\beta\}$, and by Lemma~\ref{lem:multi_fan}$(a)$, $\gamma \neq \alpha$. If $\gamma \in \varphi(y_\beta)$, we are done. Hence $\gamma \notin \varphi(y_\beta)$, and we consider the chain $P = P_x(\beta,\gamma,\varphi)$ and the coloring $\varphi' = \varphi/P$. 
Note that $\varphi' \in \mathcal{C}^{\Delta}(G - xy)$ satisfies $\beta \in \overline{\varphi'}(x)$, $\alpha \in \overline{\varphi'}(y)$, and $\varphi'(xz) = \alpha$.

If $\varphi'(yy_\beta) = \varphi(yy_\beta) = \beta$, then Lemma~\ref{lem:multi_fan}$(a)$ implies that $y_\beta \in K(x,y)$ and we are done. Assume $\varphi'(yy_\beta) \neq \varphi(yy_\beta)$. If $\varphi'(yy_\beta) \ne \beta$, then $yy_\beta \in E(P)$. 
Note that $\gamma \in \overline{\varphi}(y_\beta) \cap \overline{\varphi}(z_\beta)$, which implies that $y_\beta$ is the other endvertex of $P$, and so $zz_\beta \notin E(P)$. Hence $\varphi'(zz_\beta) = \varphi(zz_\beta) = \beta$. Since $\beta \in \overline{\varphi'}(x)$ and $\varphi'(xz) = \alpha \in \overline{\varphi'}(y)$, Lemma~\ref{prop: 25} implies that $z_\beta \in K(x,z)$, a contradiction. This completes the proof of $(a)$. Statement $(b)$ follows once we uncolor $xz$ and color $xy$ with $\alpha$.
\end{proof}

\begin{claim}\label{claim:3}
Let $e=xy \in E_G(x, y)$, and let $\varphi \in \mathcal{C}^{\Delta}(G - e)$. Let $z = x_{\alpha,\varphi}$, where $\alpha \in \overline{\varphi}(y)$, and let $\beta \in \varphi(x)\setminus \{\alpha\}$.
Assume that, whenever defined, neither $y_\beta$ nor $z_\beta$ is equal to $x$. Then $y_\beta \in K(x,y)$ or $z_\beta \in K(x,z)$.
\end{claim} 
\begin{proof}
By Claim~\ref{claim:1}, we may assume that $\beta \in \varphi(y)\cap \varphi(z)$, as otherwise $y_\beta \in K(x,y)$ or $z_\beta \in K(x,z)$. Thus $y_\beta, z_\beta$ are well-defined and $x \notin \{y_\beta, z_\beta\}$. Then, the vertices $x,y,z,x_\beta,y_\beta,z_\beta$ are distinct.
Suppose, on the contrary, that $y_\beta \notin K(x,y)$ and $z_\beta \notin K(x,z)$. 
We may assume that $\alpha \in \overline{\varphi}(y)\cap \overline{\varphi}(y_\beta)$. Indeed, if $\alpha \in \varphi(y_\beta)$, then since $\beta \in \varphi(x)\cap \varphi(y)\cap \varphi(z)$, by Claim~\ref{claim:2}$(b)$, there exists a color $\gamma' \in \overline{\varphi}(x)\cap \overline{\varphi}(y_\beta)\cap \varphi(y)\cap \varphi(z_\beta)$. Clearly, $\gamma'\notin\{\alpha,\beta\}$. By Lemma~\ref{lem:multi_fan}$(b)$, the chain $P=P_x(\alpha,\gamma',\varphi)$ is a path joining $x$ and $y$, and hence $z \in V(P)$ and $y_\beta \notin V(P)$. The chain $P'=P_{y_\beta}(\alpha,\gamma',\varphi)$ is disjoint from $P$. Interchanging colors $\alpha$ and $\gamma'$ along $P'$ yields a coloring $\varphi'$ satisfying the conditions of this claim, with $\alpha \in \overline{\varphi'}(y)\cap \overline{\varphi'}(y_\beta)$. 

Now assume that $\alpha \in \overline{\varphi}(y)\cap \overline{\varphi}(y_\beta)$.
By Claim~\ref{claim:2}, there exist colors $\gamma \in \overline{\varphi}(x)\cap \overline{\varphi}(z_\beta)\cap \varphi(z)\cap \varphi(y_\beta)$ and $\gamma' \in \overline{\varphi}(x)\cap \overline{\varphi}(y_\beta)\cap \varphi(y)\cap \varphi(z_\beta)$. In particular, $\alpha,\beta,\gamma,\gamma'$ are all distinct and $\gamma, \gamma' \in \varphi(y) \cap \varphi(z)$ by Lemma~\ref{lem:multi_fan}$(a)$. Recolor $yy_\beta$ with $\alpha$ to obtain a coloring $\varphi_1 \in \mathcal{C}^{\Delta}(G-xy)$. Then $\beta \in \overline{\varphi_1}(y)\cap \overline{\varphi_1}(y_\beta)$, and $P_x(\beta,\gamma,\varphi_1)$ is a path joining $x$ and $y$. Let $\varphi_2=\varphi_1/P_x(\beta,\gamma,\varphi_1)$. Then $\beta \in \overline{\varphi_2}(x)\cap\overline{\varphi_2}(y_\beta)$, $\gamma \in \overline{\varphi_2}(y)$, $\varphi_2(xz)=\varphi_2(yy_\beta)=\alpha$, and $\varphi_2(zz_\beta)=\beta$, as $z,z_\beta,y_\beta \notin V(P_x(\beta,\gamma,\varphi_1))$. Since $\gamma \in \overline{\varphi_2}(y)\cap \overline{\varphi_2}(z_\beta)$ and $\gamma' \in \overline{\varphi_2}(x)\cap \overline{\varphi_2}(y_\beta)$, the chain $P_1:=P_x(\gamma,\gamma',\varphi_2)$ is also a path joining $x$ and $y$. Note that $y_\beta, z_\beta \notin V(P_1)$. Interchange $\gamma$ and $\gamma'$ on all edges not on the path $P_1$, and denote the resulting coloring by $\varphi_3$. Then $\gamma \in \overline{\varphi_3}(y)\cap \overline{\varphi_3}(y_\beta)$, $\gamma' \in \overline{\varphi_3}(x)\cap \overline{\varphi_3}(z_\beta)$, $\beta \in \overline{\varphi_3}(x)\cap \overline{\varphi_3}(y_\beta)$, $\varphi_3(xz)=\varphi_3(yy_\beta)=\alpha$, and $\varphi_3(zz_\beta)=\beta$. Finally, recolor $yy_\beta$ with $\gamma$ to obtain a coloring $\varphi_4 \in \mathcal{C}^{\Delta}(G-xy)$. Then $\alpha \in \overline{\varphi_4}(y)$, $\beta \in \overline{\varphi_4}(x)$, $\varphi_4(xz)=\alpha$, and $\varphi_4(zz_\beta)=\beta$. Therefore $(y,yx,x,xz,z,zz_\beta,z_\beta)$ is a Kierstead path with respect to $xy$ and $\varphi_4$. By Lemma~\ref{lem:ShortK_Path}$(c)$, we obtain $d_G(x)+d_G(z)+d_G(z_\beta)\ge 2\Delta+2$, and hence $z_\beta \in K(x,z)$, a contradiction. This completes the proof.
\end{proof}

Let $z \in Z$. By the definition of $Z$, there is a coloring $\varphi \in \mathcal{C}^{\Delta}(G - e)$ such that some edge $f \in E_G(x,z)$ satisfies $\varphi(f) \in \overline{\varphi}(y)$. It follows from Lemmas~\ref{lem:degree}$(a)$,~\ref{prop: 25} and Claim~\ref{claim:3} that 
\begin{align*}
\sigma'(x,y) + \sigma'(x,z) &= |E_G(y, K(x,y))| + |E_G(z, K(x,z))| \\ 
&\ge 2|\overline{\varphi}(x)| + |\varphi(x)| - 1 - (\mu_G(xy)-1) - (\mu_G(xz)-1) \\ 
&= 2\Delta - d_G(x) + 2 - \mu_G(xy) -\mu_G(xz)\\
&\ge 2\Delta - d_G(x) + 2 - 2\mu.
\end{align*} 
The proof is complete. \qed

\section*{Concluding Remarks}

In this paper, we propose the conjecture that every $\Delta$-critical graph has average degree at least $(2\Delta+2)/3$. We verify this conjecture for small values of $\Delta$, and also give several lower bounds on the average degree for general $\Delta$. Our approach is based on the discharging method, together with key structural lemmas (notably Lemmas~\ref{lem:degree} and~\ref{lem: sigma}).

Although we obtain a multigraph generalization of Woodall's Lemma, it is not yet strong enough to resolve the conjecture. This suggests that further progress toward the conjecture will likely require substantially strengthening these structural lemmas and refining the discharging method.

\bibliographystyle{plain}
\bibliography{reference.bib} 
\end{document}